\documentclass[12pt,twoside]{amsart}
\usepackage{amssymb}
\usepackage{verbatim}
\usepackage{amsmath}
\usepackage{bm}
\usepackage{a4wide}
\usepackage[latin1]{inputenc}
\usepackage[T1]{fontenc}
\usepackage{times}
\usepackage{amssymb,latexsym}
\usepackage{enumerate}
\usepackage{color}

\makeatletter
\newcommand{\sumprime}{\if@display\sideset{}{'}\sum%
            \else\sum'\fi}
\makeatother

\begin{document}

\numberwithin{equation}{section}

\newtheorem{theorem}{Theorem}[section]
\newtheorem{proposition}[theorem]{Proposition}
\newtheorem{conjecture}[theorem]{Conjecture}
\def\theconjecture{\unskip}
\newtheorem{corollary}[theorem]{Corollary}
\newtheorem{lemma}[theorem]{Lemma}
\newtheorem{observation}[theorem]{Observation}
\newtheorem{definition}{Definition}
\numberwithin{definition}{section} 
\newtheorem{remark}{Remark}
\def\theremark{\unskip}
\newtheorem{kl}{Key Lemma}
\def\thekl{\unskip}
\newtheorem{question}{Question}
\def\thequestion{\unskip}
\newtheorem{example}{Example}
\def\theexample{\unskip}
\newtheorem{problem}{Problem}

\thanks{Supported by National Natural Science Foundation of China, No. 12271101}

\address [Bo-Yong Chen] {School of Mathematical Sciences,  Fudan University,  Shanghai, 200433, China}
\email{boychen@fudan.edu.cn}

\title{The theorems of M.  Riesz and Zygmund  in several complex variables}
\author{Bo-Yong Chen}

\date{}

\begin{abstract}
In this note,  we extend the well-known theorems of M.  Riesz and Zygmund  on conjugate functions as follows.  Let $\Omega$ be a  domain  in $\mathbb C^n$.   Suppose that $f=u+iv\in \mathcal O(\Omega)$ satisfies $v(z_0)=0$\/ for some $z_0\in \Omega$.  Then $
\|f\|_{p,z_0} \le C_p\, \|u\|_{p,z_0}
$ for $1<p<\infty$,  
where  $C_p$ is a constant  depending only on $p$ and 
$
\|u\|_{p,z_0}^p
$
is defined to be the value at $z_0$ of the least harmonic majorant of $|u|^p$.   
Moreover,   if $|u|\le 1$,  then for any $\alpha>1$,  there exists $C_\alpha>0$ such that $ \int_{\partial \Omega_t} \frac{\exp\left(\frac{\pi}2 |f| \right)}{(1+|f|)^\alpha}\, d\omega_{z_0,t} \le C_\alpha$ for any exhaustion $\{\Omega_t\}$ of $\Omega$ with $\Omega_t\ni z_0$, where $d \omega_{z_0,t}$ is the harmonic measure of $\Omega_t$ relative to $z_0$.   
Analogous results for  Poletsky-Stessin-Hardy spaces on hyperconvex domains are given.   
\end{abstract}

\maketitle

\section{Introduction}

We start with  the  celebrated result of M.  Riesz on conjugate functions in Hardy spaces.

\begin{theorem}[cf.  \cite{Riesz}]\label{th:Riesz}
Let $f=u+iv$ is a holomorphic function on the unit disc $\mathbb D$ normalized by  $v(0)=0$.     
If $u\in h^p$  for some $1<p<\infty$,  then $f\in H^p$,  such that $\|f\|_{H^p}\le C_p \|u\|_{h^p}$,  where $C_p$ is a constant depending only on $p$.   Here $h^p$ and $H^p$ denote the harmonic and holomorphic Hardy space respectively.   
\end{theorem}

It is well-known that Riesz's theorem breaks down for $p=\infty$.  Nevertheless,  Zygmund proved the following

\begin{theorem}[cf.  \cite{Zygmund}]
If $f=u+iv\in \mathcal O(\mathbb D)$ satisfies $|u|\le 1$ and $v(0)=0$,  then for any $\gamma<\pi/2$,   there exists a constant $C_\gamma>0$ such that 
$$
\int_0^{2\pi} \exp\left(\gamma \left|f(r e^{i\theta})\right| \right) d\theta \le C_\gamma,\ \ \ \forall\,r<1.
$$
\end{theorem}

The goal of this note is to extend these results to Hardy spaces for arbitrary domains in $\mathbb C^n$.   Let $\Omega$ be  a domain in $\mathbb C^n$.   For $0<p<\infty$,   the harmonic Hardy space $h^p(\Omega)$ is defined to be the set of all harmonic functions $u$ on $\Omega$ such that $|u|^p$ admits a harmonic majorant in $\Omega$.  For any $z_0\in \Omega$,  one may define a norm on $h^p(\Omega)$ for $1\le p<\infty$ by
$$
\|u\|_{p,z_0}:=\left( \inf \left\{ U(z_0): U \text{ is harmonic on } \Omega,  |u|^p\le U \right\}\right)^{1/p}.
$$
The holomorphic Hardy space $H^p(\Omega)$ is the set of all holomorphic functions $f$ on $\Omega$ that both ${\rm Re\,}f,{\rm Im\,}f$ are contained in  $h^p(\Omega)$.   An equivalent definition of $\|\cdot\|_{p,z_0}$ is given as follows.  
Let $\Omega$ be exhausted by an increasing family $\{\Omega_t\}$ of relatively compact domains,  containing $z_0$.   
It is well-known that  for any harmonic function $u$ on $\Omega$,  
\begin{equation}\label{eq:Lumer}
\|u\|_{p,z_0}=\sup_t \left[\int_{\partial \Omega_t} |u|^p d\omega_{z_0,t}  \right]^{1/p}
\end{equation}
(see \cite{Lumer},  Chapter 2).   Here $d\omega_{z_0,t}$  denotes the harmonic measure of $\Omega_t$ relative to $z_0$. 
Although the norms depend on the choice of $z_0$,  they are mutually equivalent. 
 If $\Omega$ is a bounded domain with $C^2-$boundary,  then 
$$
\|u\|_{p,z_0} = \left[ \int_{\partial \Omega} |u|^p  P_\Omega(z_0,\cdot) dS \right]^{1/p}
$$
for $1<p<\infty$,  
where  $P_\Omega$ denotes the Poisson kernel of\/ $\Omega$ and $dS$ the surface element on $\partial \Omega$.   Since $C^{-1}\le P_\Omega(z_0,\cdot)|_{\partial \Omega}\le C$ for suitable  constant $C>0$,  it follows that $h^p(\Omega)$ and $H^p(\Omega)$ coincide with  Hardy spaces  in the usual sense (cf.  Stein \cite{Stein}).

In the 1970's,  Stout proved the following

\begin{theorem}[cf.  \cite{Stout}]\label{th:Stout}
Let $\Omega$ be a bounded domain with $C^2-$boundary in $\mathbb C^n$.   Suppose that $f=u+iv\in \mathcal O(\Omega)$ satisfies $v(z_0)=0$\/ for some $z_0\in \Omega$ and  $u\in h^p(\Omega)$\/ for some $1<p<\infty$.  Then $f\in H^p(\Omega)$ and
$$
\|f\|_{p,z_0} \le C_{p,\Omega}\, \|u\|_{p,z_0}
$$
where  $C_{p,\Omega}$ is a constant  depending only on $p$ and $\Omega$.
\end{theorem}

\begin{remark}
Earlier,  Range-Siu \cite{RangeSiu} proved the theorem for special values $p=2,4,6,\cdots$,  and the constant can be chosen independent of\/ $\Omega$.  
\end{remark}

Our first result is the following 

\begin{theorem}\label{th:Main}
Let $\Omega$ be a  domain  in $\mathbb C^n$.   Suppose that $f=u+iv\in \mathcal O(\Omega)$ satisfies $v(z_0)=0$\/ for some $z_0\in \Omega$ and  $u\in h^p(\Omega)$\/ for some $1<p<\infty$.  Then $f\in H^p(\Omega)$ and
\begin{equation}\label{eq:Main}
\|f\|_{p,z_0} \le C_p\, \|u\|_{p,z_0}
\end{equation}
where  $C_p$ is a constant  depending only on $p$.
\end{theorem}

In \cite{Stout},  Stout deduced  Theorem \ref{th:Stout} from Theorem \ref{th:Riesz} by using the slicing technique and Riemann mappings in a rather delicate way,  and the $C^2$ condition of $\partial \Omega$ is essential in his proof.   Our approach is more direct: The case $1<p\le 2$ follows by a straightforward application of  the  Poisson-Jensen formula (which is similar as Duren \cite{Duren}),  with $C_p=(\frac{p}{p-1})^{1/p}$;  while  the case $2<p<\infty$ can be derived from  the case $1<p\le 2$,  but with a worse constant.  Moreover,  a similar argument as the case $1<p\le 2$  yields the following

\begin{theorem}\label{th:Zygmund}
Let $\Omega$ be a  domain  in $\mathbb C^n$.   If $f=u+iv\in \mathcal O(\Omega)$ satisfies $|u|\le 1$ and $v(z_0)=0$ for some $z_0\in \Omega$,  then for any $\alpha>1$,  there exists a constant $C_\alpha>0$ such that 
\begin{equation}\label{eq:Zygmund}
\int_{\partial \Omega_t} \frac{\exp\left(\frac{\pi}2 |f| \right)}{(1+|f|)^\alpha}\, d\omega_{z_0,t} \le C_\alpha
\end{equation}  
holds for  any exhaustion $\{\Omega_t\}$ of\/ $\Omega$ with $\Omega_t\ni z_0$.
\end{theorem}

\begin{remark}
The function $f(z)=\frac{2}{\pi i} \log \frac{1+z}{1-z}$ is holomorphic on $\mathbb D$ and satisfies $|{\rm Re\,}f| \le 1$,  but a simple calculation shows 
$$
\int_0^{2\pi} \frac{\exp\left( \frac{\pi}2 |f(e^{i\theta})|\right)}{\log (|f(e^{i\theta})|+1)}\, d\theta=\infty.
$$
\end{remark}

There is another  generalization of Hardy spaces  due to Poletsky-Stessin \cite{PS}.  Let $\Omega$ be a hyperconvex domain in $\mathbb C^n$,  i.e.,  there exists a negative continuous plurisubharmonic (psh) function $\rho$ on $\Omega$ such that $\{\rho<c\}\subset\subset \Omega$ for all $c<0$.  The pluricomplex Green function  on $\Omega$ is defined by
$$
g_\Omega(\zeta,z)=\sup \left\{u(\zeta): u\in PSH^-(\Omega), \limsup_{\zeta\rightarrow z}(u(\zeta)-\log|\zeta-z|)<\infty \right\}
$$ 
where $PSH^-(\Omega)$ denotes the set of negative psh functions on $\Omega$.  Given $z_0 \in \Omega$ and $t>0$,  define 
$$
\Omega_t:=\{z\in \Omega: g_\Omega(\cdot,z_0)<-t\},\ \ \ g_{z_0,t}(z):=\max\{-t,g_\Omega(\cdot,z_0)\}. 
$$ 
According to Demailly \cite{Demailly},  $g_\Omega(z,w)$ is continuous on $\overline{\Omega}\times \Omega \setminus \{z=w\}$ and maximal on $\Omega\setminus \{w\}$;  moreover,  
 the so-called Lelong-Jensen formula gives
\begin{equation}\label{eq:Lelong-Jensen}
\int_\Omega \varphi (dd^c g_{z_0,t})^n -(2\pi)^n \varphi(z_0) = \int_{\Omega_t} (-t-g_\Omega(\cdot,z_0)) dd^c \varphi \wedge (dd^c g_\Omega(\cdot,z_0))^{n-1}
\end{equation}
for any continuous psh function $\varphi$ on $\Omega$.  Here $d\mu_{z_0,t}:=(dd^c g_{z_0,t})^n$ is called the pluriharmonic measure of $\Omega_t$ relative to $z_0$,   which is supported on $\partial \Omega_t$ with total measure $(2\pi)^n$.  The pluriharmonic  Poletsky-Stessin-Hardy space $h^p_{PS}(\Omega)$ is the set of all pluriharmonic functions $u$ on $\Omega$ such that
$$
\|u\|_{p,z_0}^{PS}:=\sup_{t>0}\left[ \int_{\Omega_t} |u|^p d\mu_{z_0,t} \right]^{1/p}<\infty.
$$
The holomorphic Poletsky-Stessin-Hardy space $H^p_{PS}(\Omega)$ is the set of all holomorphic functions $f$ on $\Omega$ such that both ${\rm Re\,}f,{\rm Im\,}f$ are contained in $h^p_{PS}(\Omega)$.   

We have the following analogue of Theorem \ref{th:Main} and Theorem \ref{th:Zygmund} respectively.  

\begin{theorem}\label{th:Main_2}
Let $\Omega$ be a hyperconvex domain  in $\mathbb C^n$.   Suppose that $f=u+iv\in \mathcal O(\Omega)$ satisfies $v(z_0)=0$\/ for some $z_0\in \Omega$ and  $u\in h^p_{PS}(\Omega)$\/ for some $1<p<\infty$.  Then $f\in H^p_{PS}(\Omega)$ and
\begin{equation}\label{eq:Main_2}
\|f\|_{p,z_0}^{PS} \le C_p\, \|u\|_{p,z_0}^{PS}
\end{equation}
where  $C_p$ is a constant  depending only on $p$.
\end{theorem}

\begin{theorem}\label{th:Z_2}
Let $\Omega$ be a hyperconvex domain  in $\mathbb C^n$.   If $f=u+iv\in \mathcal O(\Omega)$ satisfies $|u|\le 1$ and $v(z_0)=0$ for some $z_0\in \Omega$,  then for any $\alpha>0$,  there exists a constant $C_\alpha>0$ such that 
$$
\int_{\partial \Omega_t} \frac{\exp\left(\frac{\pi}2 |f| \right)}{(1+|f|)^\alpha}\, d\mu_{z_0,t} \le C_\alpha.
$$  
\end{theorem}

\section{Proof of Theorem \ref{th:Main}}

We need the following elementary lemma.

\begin{lemma}\label{lm:Riesz-type}
For  $1<p\le 2$ and $\tau>0$,    $\frac{p}{p-1}(u^2+\tau)^{p/2}-(|f|^2+\tau)^{p/2}$ is psh on $\Omega$.  
\end{lemma}

\begin{proof}
A straightforward calculation shows 
\begin{eqnarray}\label{eq:RT_0}
\partial\bar{\partial} (|f|^2+\tau)^{p/2} & = & \frac{p}2 (|f|^2+\tau)^{\frac{p}2-1}\partial\bar{\partial} |f|^2 + \frac{p}2\left(\frac{p}2-1\right)(|f|^2+\tau)^{\frac{p}2-2}\partial |f|^2\wedge \bar{\partial} |f|^2\nonumber\\
& = & \frac{p^2}4 (|f|^2+\tau)^{\frac{p}2-1} \partial f\wedge \overline{\partial f} + \frac{p}2\left(1-\frac{p}2\right) \tau (|f|^2+\tau)^{\frac{p}2-2}\partial f\wedge \overline{\partial f}, 
\end{eqnarray}
and since $\partial\bar\partial u=0$,
\begin{eqnarray}\label{eq:RT_0'}
\partial\bar{\partial} (u^2+\tau)^{p/2} & = & \frac{p}2 (u^2+\tau)^{\frac{p}2-1}\partial\bar{\partial} u^2 + \frac{p}2\left(\frac{p}2-1\right)(u^2+\tau)^{\frac{p}2-2}\partial u^2\wedge \bar{\partial} u^2\nonumber\\
& = & p(p-1) (u^2+\tau)^{\frac{p}2-1} \partial u\wedge \bar \partial u + p(2-p) \tau (u^2+\tau)^{\frac{p}2-2} \partial u\wedge \bar{\partial} u.
\end{eqnarray}
Since $f$ is holomorphic,  we have $\partial u=\frac12 \partial f$,  so that
$$
\partial u\wedge \bar{\partial} u=\frac14 \partial f\wedge \overline{\partial f}.
$$
Thus \eqref{eq:RT_0} together with \eqref{eq:RT_0'} give
\begin{eqnarray*}
\frac{p}{p-1} i\partial \bar{\partial} (u^2+\tau)^{p/2} & \ge & p^2 (u^2+\tau)^{\frac{p}2-1}i\partial u\wedge \bar{\partial} u + p(2-p) \tau (u^2+\tau)^{\frac{p}2-2} i\partial u\wedge \bar{\partial} u\\
& \ge & i\partial\bar\partial (|f|^2+\tau)^{p/2}.
\end{eqnarray*}
\end{proof}

\begin{proof}[Proof of Theorem \ref{th:Main}]
Step 1.  Let $1<p\le 2$.  Let $G_{\Omega_t}(\cdot,z_0)$ be the (negative) Green function of $\Omega_t$ with pole at $z_0$,  so that $\Delta G_{\Omega_t}(\cdot,z_0)=\delta_{z_0}$,  the Dirac measure at $z_0$.  By the Poisson-Jensen formula (cf.  \cite{Hayman},  Theorem 5.27 and Theorem 5.23;  see also \cite{Ransford} for $n=1$),  we have
\begin{eqnarray*}
\int_{\partial \Omega_t} (|f|^2+\tau)^{p/2} d\omega_{z_0,t} 
& = & (|f(z_0)|^2+\tau)^{p/2} - \int_{\Omega_t}  G_{\Omega_t}(\cdot,z_0) \Delta  (|f|^2+\tau)^{p/2}\\
& \le &\frac{p}{p-1} \left[(u(z_0)^2+\tau)^{p/2} -  \int_{\Omega_t}  G_{\Omega_t}(\cdot,z_0) \Delta  (u^2+\tau)^{p/2}\right]\\
& = & \frac{p}{p-1}\int_{\partial \Omega_t} (u^2+\tau)^{p/2} d\omega_{z_0,t}
\end{eqnarray*}
where the inequality follows from Lemma \ref{lm:Riesz-type} and the condition $v(z_0)=0$ (so that $f(z_0)=u(z_0)$).  

Letting $\tau\rightarrow 0+$,  we  obtain
$$
\int_{\partial \Omega_t} |f|^p d\omega_{z_0,t}  \le \frac{p}{p-1} \int_{\partial \Omega_t} |u|^p d\omega_{z_0,t}.
$$
This combined with \eqref{eq:Lumer} gives \eqref{eq:Main} with $C_p=\left(\frac{p}{p-1}\right)^{1/p}$.   

Step 2.  Let $2<p\le 4$.  Set $\tilde{p}=p/2$ and $\tilde{f}=f^2$.  If one writes $\tilde{f}=\tilde{u}+i\tilde{v}$,  then
$$
\tilde{u}= u^2-v^2,\ \ \ \tilde{v}=2uv.
$$
Set $\tilde{g}=i(\tilde{f}-\tilde{u}(z_0))= i(\tilde{f}-{u}(z_0)^2)$.   Since $\mathrm{Im\,}\tilde{g}(z_0)=0$ and $1<\tilde{p}<2$,  it follows from Step 1 that 
$$
\int_{\partial \Omega_t} |\tilde{g}|^{\tilde{p}} d\omega_{z_0,t} \le \frac{\tilde{p}}{\tilde{p}-1}
\int_{\partial \Omega_t} |\tilde{v}|^{\tilde{p}} d\omega_{z_0,t} = \frac{\tilde{p}\, 2^{\tilde{p}}}{\tilde{p}-1}
\int_{\partial \Omega_t} |uv|^{\tilde{p}} d\omega_{z_0,t}. 
$$  
Thus 
\begin{eqnarray*}
\int_{\partial \Omega_t} |f|^{{p}} d\omega_{z_0,t} & = & \int_{\partial \Omega_t} |\tilde{f}|^{\tilde{p}} d\omega_{z_0,t} \le 2^{\tilde{p}-1} \left[ |{u}(z_0)|^{{p}} 
+ \int_{\partial \Omega_t} |\tilde{g}|^{\tilde{p}} d\omega_{z_0,t}\right] \\
& \le &  2^{\tilde{p}-1} \left[  \int_{\partial \Omega_t} |u|^{{p}} d\omega_{z_0,t}
+  \frac{\tilde{p}\, 2^{\tilde{p}}}{\tilde{p}-1}
\int_{\partial \Omega_t} |uv|^{\tilde{p}} d\omega_{z_0,t}
\right].
\end{eqnarray*}
Use the elementary inequality $2|uv|^{\tilde{p}}\le \varepsilon^{-1} |u|^p + \varepsilon |v|^p$ with $\varepsilon\ll1$,  we conclude that  
$$
\int_{\partial \Omega_t} |f|^{{p}} d\omega_{z_0,t} \le C_p  \int_{\partial \Omega_t} |u|^{{p}} d\omega_{z_0,t}
$$
for suitable $C_p>0$.  

Step 3.   Let $3<p<\infty$.  We define  $p'$ to be the largest odd number which is smaller than $p$. Set $\tilde{p}=p/p'$ and $\tilde{f}=f^{p'}$.  
Since $p'\ge p-2$ and $p'\ge 3$,  it follows that
$$
1< \tilde{p}\le 1+\frac2{p'} <2.
$$
  Write $\tilde{f}=\tilde{u}+i\tilde{v}\in \mathcal O(\Omega)$.  Then we have
$$
\tilde{u}=\sum_{k: 2k < p'} \tbinom{p'}{2k}  (iv)^{2k}u^{p'-2k}\ \ \ \text{and}\ \ \  i\tilde{v}=\sum_{k: 2k-1\le p'} \tbinom{p'}{2k-1}  (iv)^{2k-1}u^{p'-2k+1}.
$$
Since $v(z_0)=0$,   so $\tilde{v}(z_0)=0$.  Moreover,  we have
\begin{eqnarray}\label{eq:RT_1}
\int_{\partial \Omega_t} |f|^p d\omega_{z_0,t}  & = & \int_{\partial \Omega_t} |\tilde{f}|^{\tilde{p}} d\omega_{z_0,t} \le \frac{\tilde{p}}{\tilde{p}-1} \int_{\partial \Omega_t} |\tilde{u}|^{\tilde{p}} d\omega_{z_0,t}  \nonumber\\
& \lesssim &  \sum_{k: 2k < p'} \int_{\partial \Omega_t} |v|^{2k \tilde{p}}|u|^{p-2k\tilde{p}} d\omega_{z_0,t} 
\end{eqnarray}
where the implicit constant depends only on $p$.  

On the other hand,  H\"older's inequality yields
\begin{eqnarray}\label{eq:RT_2}
&& \int_{\partial \Omega_t} |v|^{2k \tilde{p}}|u|^{p-2k\tilde{p}} d\omega_{z_0,t} \nonumber\\
& \le & \left[ \int_{\partial \Omega_t} |v|^p d\omega_{z_0,t}   \right]^{2k/p'}
 \left[ \int_{\partial \Omega_t} |u|^p d\omega_{z_0,t} \right]^{1-2k/p'}\nonumber\\
 & \le & \left[ \int_{\partial \Omega_t} |f|^p d\omega_{z_0,t}  \right]^{2k/p'}
 \left[ \int_{\partial \Omega_t} |u|^p d\omega_{z_0,t}  \right]^{1-2k/p'}\nonumber\\
 & \le & \frac{\varepsilon^{p'/2k}}{p'/2k} \int_{\partial \Omega_t} |f|^p d\omega_{z_0,t}  + \frac{\varepsilon^{-p'/(p'-2k)}}{p'/(p'-2k)} \int_{\partial \Omega_t} |u|^p d\omega_{z_0,t}
\end{eqnarray}
for all $\varepsilon>0$ and $0<k < p'/2$,  in view of Young's inequality.   
\eqref{eq:RT_1} and \eqref{eq:RT_2} 
yield \eqref{eq:Main} provided  $\varepsilon\ll1$.  
\end{proof}

\begin{remark}
In Step 2,  the constant $C_p$ tends to infinity as $p\rightarrow 2+$.  It would be interesting to find certain $C_p$ which is uniformly bounded for $2<p\le 3$.  
\end{remark}

\section{Proof of Theorem \ref{th:Zygmund}}

We first show  the following 

\begin{proposition}\label{prop:Kolmogorov}
Let $\Omega$ be a  domain  in $\mathbb C^n$ and $z_0\in \Omega$.   If\/ $f=u+iv\in \mathcal O(\Omega)$ satisfies $u>0$,  then for any $\alpha>1$ and  any exhaustion $\{\Omega_t\}$ of\/ $\Omega$ with $\Omega_t\ni z_0$,  
$$
\int_{\partial \Omega_t} \frac{|f +e^{\alpha+1}|}{\log^\alpha |f+e^{\alpha+1}|^2 }\, d\omega_{z_0,t} < \frac{|f(z_0) +e^{\alpha+1}|}{\log^\alpha |f(z_0)
 + e^{\alpha+1}|^2} + 
 \frac{\alpha+1}{\alpha-1}\,  \frac{u(z_0) +e^{\alpha+1}}{\log^{\alpha-1} (u(z_0) +e^{\alpha+1})^2}.
$$
\end{proposition}

\begin{proof}
Given $\beta>0$,  define 
$$
\lambda_\beta(t):=\exp\left(\frac{t}2-\beta \log t\right),\ \ \ t>0.
$$
Then we have
\begin{eqnarray*}
\lambda_\beta'(t) & = &\lambda_\beta(t)\left(\frac12-\frac{\beta}t \right)\\
\lambda_\beta''(t) & = & \lambda_\beta(t)\left(\frac12-\frac{\beta}t\right)^2 + \lambda_\beta(t)\,\frac{\beta}{t^2}.  
\end{eqnarray*}
For $t\ge 2(\alpha+1)$,  we have 
$$
\lambda_\alpha''(t)\le \lambda_\alpha(t)\left(\frac14 + \frac{\alpha(\alpha+1)}{t^2}\right)\le \frac{\lambda_\alpha(t)}2.
$$
Replace $f$ by $f+e^{\alpha+1}$,  we always assume $|f|\ge u\ge e^{\alpha+1}$.  Set 
$$
\varphi=\lambda_\alpha(\log |f|^2),\ \ \ \psi=\lambda_{\alpha-1}(\log u^2).
$$
Since $\log |f|^2$ is pluriharmonic on $\Omega$,  we have
\begin{eqnarray*}
i\partial\bar{\partial} \varphi & = & \lambda_\alpha ''(\log |f|^2) i\partial \log |f|^2\wedge \bar{\partial} \log |f|^2\\
& \le & \frac{\lambda_\alpha(\log |f|^2)}2  \frac{i\partial f\wedge \overline{\partial f} }{|f|^2 }\\
& = &
 \frac{i\partial f\wedge \overline{\partial f} }{2|f| \log^\alpha |f|^2}.
\end{eqnarray*}
On the other hand,  since $u$ is pluriharmonic,  we have $\partial \bar{\partial} \log u^2=-2u^{-2} \partial u\wedge \bar{\partial} u$,  so that
\begin{eqnarray*}
\partial \bar{\partial} \psi & = & \lambda_{\alpha-1}'(\log u^2) \partial \bar{\partial} \log u^2 +
\lambda_{\alpha-1}''(\log u^2) \partial \log u^2\wedge \bar{\partial} \log u^2\\
& = & \left(-\frac{\lambda_{\alpha-1}'(\log u^2)}2 +\lambda_{\alpha-1}''(\log u^2) \right) \frac{4\partial u\wedge \bar{\partial} u}{u^2}\\
& = & \left(-\frac{\alpha-1}2+\frac{\alpha(\alpha-1)}{\log u^2}\right) \frac{4\partial u\wedge \bar{\partial}u}{u\log^\alpha u^2}.
\end{eqnarray*}
Note that $\partial u=\frac12 \partial f$,  $|f|\ge u\ge e^{\alpha+1}$.  Thus
$$
-i\partial \bar{\partial} \psi \ge \frac{\alpha-1}{2(\alpha+1)} \frac{i \partial f\wedge \overline{\partial f} }{|f| \log^\alpha |f|^2}.
$$
It follows that 
$$
\phi:=-\frac{\alpha+1}{\alpha-1}\psi-\varphi
$$
is a smooth psh function on $\Omega$.  By the Poisson-Jensen formula,  
$$
\int_{\partial \Omega_t} \phi d\omega_{z_0,t}= \phi(z_0) - \int_{\Omega_t} G_{\Omega_t}(\cdot,z_0)\Delta \phi\ge \phi(z_0).
$$ 
Hence
$$
\int_{\partial \Omega_t} \varphi d\omega_{z_0,t}\le -\phi(z_0)-\frac{\alpha+1}{\alpha-1} \int_{\partial \Omega_t} \psi d\omega_{z_0,t}<  -\phi(z_0),
$$
i.e.,  
$$
\int_{\partial \Omega_t} \frac{|f|}{\log^\alpha |f|^2}\, d\omega_{z_0,t} <
\frac{|f(z_0)|}{\log^\alpha |f(z_0)|^2}+ 
\frac{\alpha+1}{\alpha-1}\,
 \frac{u(z_0)}{\log^{\alpha-1} u(z_0)^2}.
$$
\end{proof}

\begin{proof}[Proof of Theorem \ref{th:Zygmund}]
We first  assume  $|u|<1$.  Note that the conformal mappings $\exp\left( \pm \frac{\pi i}2 \zeta \right)$ map the strip $\{\zeta\in \mathbb C: |\mathrm{ Re}\,\zeta| < 1\}$ onto the right half plane.  Apply Proposition \ref{prop:Kolmogorov} to $F_{\pm}:=\exp\left( \pm \frac{\pi i}2 f \right)$,  we obtain  
$$
\int_{\partial \Omega_t} \frac{|F_{\pm} +e^{\alpha+1}|}{\log^\alpha |F_{\pm}+e^{\alpha+1}|^2 }\, d\omega_{z_0,t} \le C_\alpha\,  \frac{|F_{\pm}(z_0) +e^{\alpha+1}|}{\log^{\alpha-1} |F_{\pm}(z_0)
 + e^{\alpha+1}|^2}
$$
where $C_\alpha$ denotes a generic constant depending only on $\alpha$.  
Since $\mathrm{Re\,}F_{\pm}>0$,  we have 
$$
 e^{\mp \frac{\pi} 2 v} = |F_{\pm}|< |F_{\pm} +e^{\alpha+1}|\le |F_{\pm}| + e^{\alpha+1}
 = e^{\mp \frac{\pi} 2 v} + e^{\alpha+1},
 $$
 so that 
 $$
\int_{\partial \Omega_t} \frac{e^{\mp \frac{\pi} 2 v}  }{\log^\alpha | e^{\mp \frac{\pi} 2 v} +e^{\alpha+1}| }\, d\omega_{z_0,t} \le C_\alpha
$$
(recall that $v(z_0)=0$).  
Thus
\begin{eqnarray*}
 && \int_{\partial \Omega_t} \frac{e^{ \frac{\pi} 2 |v|}  }{\log^\alpha ( e^{ \frac{\pi} 2 |v|} +e^{\alpha+1} ) }\, d\omega_{z_0,t} \\
 & = & \int_{\partial \Omega_t \cap \{v\ge 0\}} \frac{e^{ \frac{\pi} 2 v }  }{\log^\alpha ( e^{ \frac{\pi} 2 v } +e^{\alpha+1} ) }\, d\omega_{z_0,t} 
 +   \int_{\partial \Omega_t \cap \{v < 0\}} \frac{e^{- \frac{\pi} 2 v }  }{\log^\alpha ( e^{ - \frac{\pi} 2 v } +e^{\alpha+1} ) }\, d\omega_{z_0,t}\\
 & \le & C_\alpha,
\end{eqnarray*}
from which \eqref{eq:Zygmund} immediately follows.  

If $|u|\le 1$,  then it suffices to apply the previous assertion to $r f$ with $r\rightarrow 1-$.
\end{proof}

\section{Proofs of Theorem \ref{th:Main_2} and Theorem \ref{th:Z_2}}

\begin{proof}[Proof of Theorem \ref{th:Main_2} ]
First consider the case $1<p\le 2$.  By \eqref{eq:Lelong-Jensen} and Lemma \ref{lm:Riesz-type},  we have
\begin{eqnarray*}
&& \int_{\partial \Omega_t} (|f|^2+\tau)^{p/2} d \mu_{z_0,t}
 - (2\pi)^n (|f(z_0)|^2+\tau)^{p/2}\\
 & = &  \int_{\Omega_t} (-t- g_{\Omega}(\cdot,z_0))   dd^c  (|f|^2+\tau)^{p/2}  \wedge (dd^c g_\Omega(\cdot,z_0))^{n-1}\\
& \le &\frac{p}{p-1} \int_{\Omega_t} (-t- g_{\Omega}(\cdot,z_0))   dd^c  (|u|^2+\tau)^{p/2}  \wedge (dd^c g_\Omega(\cdot,z_0))^{n-1}  \\
& = & \frac{p}{p-1} \left[ \int_{\partial \Omega_t} (|u|^2+\tau)^{p/2} d \mu_{z_0,t}
 - (2\pi)^n (|u(z_0)|^2+\tau)^{p/2} \right].  
\end{eqnarray*}
Letting $\tau\rightarrow 0+$,  we immediately obtain
$$
\int_{\partial \Omega_t} |f|^p d\mu_{z_0,t}  \le \frac{p}{p-1} \int_{\partial \Omega_t} |u|^p d\mu_{z_0,t},
$$
from which \eqref{eq:Main_2} immediately follows.

The argument of the case $2<p<\infty$ is parallel to Theorem \ref{th:Main}.
\end{proof}

\begin{proposition}\label{prop:K_2}
Let $\Omega$ be a hyperconvex domain  in $\mathbb C^n$.   Suppose that $f=u+iv\in \mathcal O(\Omega)$ satisfies $u>0$ and $v(z_0)=0$\/ for some $z_0\in \Omega$.   Then for any $\alpha>1$,
$$
\frac1{(2\pi)^n}\int_{\partial \Omega_t} \frac{|f +e^{\alpha+1}|}{\log^\alpha |f+e^{\alpha+1}|^2 }\, d\mu_{z_0,t} < \frac{|f(z_0) +e^{\alpha+1}|}{\log^\alpha |f(z_0)
 + e^{\alpha+1}|^2} + 
 \frac{\alpha+1}{\alpha-1}\,  \frac{u(z_0) +e^{\alpha+1}}{\log^{\alpha-1} (u(z_0) +e^{\alpha+1})^2}.
$$
\end{proposition}

\begin{proof}
Let $\phi$ be as  the proof of Proposition \ref{prop:Kolmogorov}.  By \eqref{eq:Lelong-Jensen},  
$$
\int_{\partial \Omega_t} \phi d\mu_{z_0,t}=(2\pi)^n  \phi(z_0) + 
\int_{\Omega_t} (-t- g_{\Omega}(\cdot,z_0))   dd^c  \phi \wedge (dd^c g_\Omega(\cdot,z_0))^{n-1}
\ge (2\pi)^n \phi(z_0).
$$ 
Hence
$$
\int_{\partial \Omega_t} \varphi d\mu_{z_0,t}\le - (2\pi)^n  \phi(z_0)-\frac{\alpha+1}{\alpha-1} \int_{\partial \Omega_t} \psi d\mu_{z_0,t}<  - (2\pi)^n \phi(z_0),
$$
i.e.,  
$$
 \frac1{(2\pi)^n}  \int_{\partial \Omega_t} \frac{|f|}{\log^\alpha |f|} d\mu_{z_0,t} < \frac{|f(z_0)|}{\log^\alpha |f(z_0)|^2}+ 
\frac{\alpha+1}{\alpha-1}\,
 \frac{u(z_0)}{\log^{\alpha-1} u(z_0)^2}.
$$
\end{proof}

\begin{proof}[Proof of Theorem \ref{th:Z_2}]
Granted Proposition \ref{prop:K_2},  the argument is parallel to Theorem \ref{th:Zygmund}.
\end{proof}

\end{document}